\documentclass[10pt]{amsart}

\title{Marked length spectrum rigidity in nonpositive curvature with singularities}

\author{David Constantine}
\address{
Wesleyan University \\
Mathematics and Computer Science Department \\
Middletown, CT 06459}

\date{\today}

\newtheorem{thm}{Theorem}
\newtheorem{theorem}{Theorem}
\newtheorem{corollary}[theorem]{Corollary}
\newtheorem{lem}[thm]{Lemma}
\newtheorem{prop}[thm]{Proposition}
\newtheorem{cor}[thm]{Corollary}

\newtheorem{ques}[thm]{Question}
\newtheorem{defn}[thm]{Definition}

\numberwithin{equation}{section}
\numberwithin{thm}{section}

\theoremstyle{remark} \newtheorem{rem}[thm]{Remark}

\def\Pb{\ifmmode{\Bbb P}\else{$\Bbb P$}\fi}
\def\Z{\ifmmode{\Bbb Z}\else{$\Bbb Z$}\fi}
\def\Q{\ifmmode{\Bbb Q}\else{$\Bbb Q$}\fi}
\def\C{\ifmmode{\Bbb C}\else{$\Bbb C$}\fi}
\def\R{\ifmmode{\Bbb R}\else{$\Bbb R$}\fi}
\def\H{\ifmmode{\Bbb H}\else{$\Bbb H$}\fi}

\usepackage{graphicx} 
\usepackage{amsmath,amssymb,amsthm,amsfonts,amscd,flafter,epsf,bm,mathrsfs}
\usepackage{mathtools}
\usepackage{graphics}
\usepackage{MnSymbol}
\usepackage{epsfig}
\usepackage{psfrag}
\usepackage{color}
\input xy
\xyoption{all}

\usepackage{fancyhdr}

\begin{document}

\begin{abstract}

Combining several previously known arguments, we prove marked length spectrum rigidity for surfaces with nonpositively curved Riemannian metrics away from a finite set of cone-type singularities with cone angles $>2\pi$. With an additional condition, we can weaken the requirement on one metric to `no conjugate points.'

\end{abstract}

\maketitle

\setcounter{secnumdepth}{1}

\setcounter{section}{0}

\section{Introduction}

Let $(M,g)$ be a compact Riemannian manifold (perhaps with a small set of singularities). For any element $[\gamma]$ of $\pi_1(M)$ there is a (not necessarily unique) shortest length (piecewise) geodesic $\gamma$ in $M$ representing $[\gamma]$. The function $l_{g}:\pi_1(M) \to \mathbb{R}$ which assigns the length of $\gamma$ to $[\gamma]$ is called the \emph{marked length spectrum} of $(M,g)$. One may then ask the following general question:

\begin{ques}[The marked length spectrum rigidity question]
To what extent does the function $l_g$ determine the geometry of $(M,g)$?
\end{ques}

\noindent This is a well-studied problem; we will note some of the the results on it below. This paper addresses it for a class of singular metrics on surfaces.

Let $S$ be a compact, connected, orientable surface. We say a metric on $S$ is a \emph{nonpositively curved cone metric} if there is a finite set of points $P\subset S$ such that on $S\setminus P$ the metric is Riemannian with nonpositive curvature, and at each point in $P$ there is a cone-type singularity with cone angle $>2\pi$.  Let $\mathscr{M}_{npc}(S)$ denote the set of nonpositively curved cone metrics on $S$. We say that a metric on $S$ is a \emph{no conjugate points cone metric} if it has a finite set of cone points $P$ with cone angles $>2\pi$ and on $S \setminus P$ the metric is Riemannian without conjugate points. We denote the set of no conjugate points cone metrics on $S$ by $\mathscr{M}_{ncpc}(S)$ and note that $\mathscr{M}_{npc}(S) \subset \mathscr{M}_{ncpc}(S)$.

A \emph{strip} is the homeomorphic image of a map from $\mathbb{R}\times [0,\epsilon]$ into $S$ where the first coordinate parametrizes unit-speed geodesics, and the lifts of these geodesics to $\tilde S$ remain at a bounded distance from one another. In non-positive curvature, such a strip is necessarily an isometric immersion -- a \emph{flat strip}. The goal of this paper is to prove the following theorem: 

\begin{theorem}\label{main thm}
Let $S$ be a surface of genus $\geq 2$. Let $g_1 \in \mathscr{M}_{ncpc}(S)$ and $g_2 \in \mathscr{M}_{npc}(S)$, and assume that $l_{g_1}=l_{g_2}$. Assume in addition that $g_1$ satisfies the following condition: 
\[ vol(\{v: v \mbox{ is tangent to a geodesic lying in a strip}\})=0 \]

\noindent where $vol$ denotes the usual volume measure on the unit tangent bundle of $S\setminus P$. Then $(S,g_1)$ is isometric to $(S,g_2)$ by an isometry isotopic to the identity.
\end{theorem}

I do not know whether the assumption on the volume of strips is redundant; i.e. whether this is true of all metrics without conjugate points. We will show below that it is satisfied if $g_1$ is in fact in $\mathscr{M}_{npc}(S)$, and in proving this we will see that the existence of cone points is not the issue, so this is a question only about non-singular metrics without conjugate points on a surface with genus $>2$.

If we restrict attention to metrics in $\mathscr{M}_{npc}(S)$, we can drop the strip volume assumption as well as the assumption on genus since the only nonpositively curved cone metrics on the torus are the flat metrics, where rigidity is clear.

\begin{corollary}
Let $S$ be any surface and let $g_1, g_2 \in \mathscr{M}_{npc}(S)$. Assume that $l_{g_1}=l_{g_2}$. Then $(S,g_1)$ is isometric to $(S,g_2)$ by an isometry isotopic to the identity.
\end{corollary}

The proof of Theorem \ref{main thm} combines ideas from the previous work of a number of authors on the marked length spectrum rigidity problem for surfaces. The general problem was posed in \cite[\S3]{bk}. Shortly thereafter it was proved for surfaces of negative curvature by Otal \cite{otal} and, independently and at roughly the same time, by Croke \cite{croke}. The methods in this paper mainly follow Otal's work. The result was extended to $g_1$ without conjugate points and $g_2$ nonpositively curved with a small (empty interior) region of zero curvature by Fathi \cite{fathi}. Croke, Fathi, and Feldman extended this result to $g_2$ of general nonpositive curvature in \cite{cff}.

Hersonsky and Paulin \cite{hp} proved MLS rigidity for negatively curved metrics with cone point singularities of angle $>2\pi$. Duchin, Leininger and Rafi proved it for metrics coming from quadratic differentials \cite{dlr} -- these are a special subset of the locally Euclidean metrics with cone points of angle $>2\pi$. Recently, Bankovic and Leininger \cite{bl} have extended this result to all piecewise Euclidean metrics with cone singularities of angle $>2\pi$. At the end of their paper they asked whether their ideas could be combined with those of Croke, Fathi, and Feldman to prove rigidity for nonpositively curved metrics with cone points. Frazier proved that the marked length spectrum distinguishes between the various curvature settings of these results \cite{frazier}.

There are a few marked-length spectrum rigidity results for non-surfaces (\cite{ham_entropy, dk, cl1d, cl_buildings}). In general, the problem is very much open for dimension greater than two.

The proof of Theorem \ref{main thm} consists of combining the ideas from this sequence of papers. We follow the approach via geodesic currents initiated by Otal. As in Croke, Fathi, and Feldman's work, much of this approach works under the no conjugate points and  nonpositive curvature assumptions on $g_1$ and $g_2$, respectively, and we follow Hersonsky and Paulin in extending Otal's ideas `measurably' to a setting with singularities. A key step in the argument is the definition of the function $\theta'(v,\theta)$ (see \S\ref{sec:angles} below) where one essentially needs to be able to detect which geodesics in the metrics contain cone points. Hersonsky and Paulin handle this using the M\"obius current, which they develop and use to prove other results; it is not clear that this can be made to work in the current setting. Instead, we use a result of Bankovic and Leininger which shows that cone points can be detected in a more `low tech' way -- by looking at the support of a certain geodesic current. These extensions of Otal's methods provide an isometry between the sets of points in the two manifolds at which the curvature is negative, or which are cone points. We then adapt some more ideas from Croke, Fathi, and Feldman, and from Bankovic and Leininger to extend the isometry to the full surface.

\subsection{Acknowledgements}
I would like to thank Jean-Fran\c cois Lafont for originally bringing Bankovic and Leininger's paper to my attention and for helpful comments on an earlier draft of this paper.


%

\section{Geodesics for metrics in $\mathscr{M}_{ncpc}(S)$}\label{sec:geodesics}

In this section we prove a few results on the structure of geodesics in $\tilde S$ for a metric with no conjugate points and with cone points of angle $>2\pi$. The results are straightforward under an assumption of nonpositive curvature, so the reader only interested in that case can skip this section apart from the definitions. They are also known in the absence of cone points, see \cite{lgreen}.

Let $\tilde g$ denote the lift of the metric $g$ to the universal cover $\tilde S$ of $S$. We will call a $\tilde g$-geodesic \emph{non-singular} if it does not hit any cone points for $\tilde g$.

We want to prove that the exponential map, where defined, is injective. This implies that any $\tilde g$-geodesic is length-minimizing, and will be used to show that $\tilde g$-geodesics do not intersect, separate, and then intersect again. For complete manifolds without conjugate points and without cone points, this is Hadamard's theorem (see, e.g. \cite[Thm 3.1]{doc}). Here we adapt aspects of that proof to deal with large-angle cone points.

Let $\exp^{\tilde g}_p$ (or $\exp_p$ when the metric is understood) denote the exponential map from $T_p\tilde S$ to $\tilde S$. This is defined for nonsingular points $p$, but it is clear that we can allow cone points if we interpret $T_p\tilde S$ as the space of positive multiples of directions from $p$. Because of the singular points, $\exp_p$ may not be defined for all vectors in $T_p\tilde S$. Let $\mathscr{V}_p$ be the set of all points in $\tilde S$ which can be reached from $p$ by a geodesic which hits no cone points (except perhaps $p$ itself). Let $V_p$ be such that $\exp_p(V_p) = \mathscr{V}_p$ and which is \emph{star-shaped}: $v\in V_p$ implies $\lambda v\in V_p$ for all $\lambda\in[0,1]$. Note that $\mathscr{V}_p$ and $V_p$ are open.

\begin{lem}\label{lem:covering}
$\exp_p:V_p\to \mathscr{V}_p$ is a covering map.
\end{lem}

\begin{proof}
Since there are no conjugate points for $\tilde g$, $\exp_p$ is a local diffeomorphism. Thus, to show it is a covering map, we just need to establish the path-lifting property. Let $c:[0,1]\to \mathscr{V}_p$ be a curve. Let $v_0\in V_p$ be any lift of $c(0)$ under $\exp_p$. Since $\exp_p$ is a local diffeomorphism, the set $A$ of all times $t$ for which $c(t)$ can be lifted to a continuous path $v(t)$ in $V_p$ starting at $v(0)=v_0$ contains at least $[0,\epsilon)$ for some $\epsilon>0$ and $A$ is open in $[0,1]$ at its right-hand endpoint. Suppose $[0,t_0)\subset A$. We wish to show that $t_0\in A$ as well, which will imply that the entire path can be lifted, proving the lemma.


Note that $\exp_p^{-1}(c)$ is a closed subset of $V_p$ and if we let $R>0$ be sufficiently large, then $c\subset \exp_p(B_R(0)\cap V_p)$.  Using the metric which has been lifted to $V_p$ we see that we can also choose $R$ so large that $v_0\in B_R(0)$ and all $v(t_n) \in B_R(0)$. Then $v(t_n)$ have an accumulation point $v^* \in \exp^{-1}_p(c) \cap B_R(0)$ because this is a closed and bounded subset of $V_p$. We extend $v(t)$ by setting $v(t_0)=v_*$. Again using the fact that $\exp_p$ is a local diffeomorphim at $v(t_0)$ we see that this is indeed a continuous extension of the lifted path. This completes the proof.
 
\end{proof}

\begin{lem}\label{lem:loc inj}
Let $\mathscr{U}$ be a simply connected subset of $\mathscr{V}_p$ containing $p$. Suppose that every point of $\mathscr{U}$ can be reached from $p$ by a geodesic in $\mathscr{U}$. Then there is a star-shaped $U\subset V_p$ containing $0$ such that $\exp_p(U) = \mathscr{U}$ and $\exp_p$ is injective on $U$.
\end{lem}

\begin{proof}
Since $\exp_p$ is a covering map on $V_p$ and $\mathscr{U}$ is simply connected, $\exp_p^{-1}(\mathscr{U})$ is a disjoint union of subsets taken diffeomorphically to $\mathscr{U}$ by $\exp_p$. Pick as $U$ the one containing $0$. It is star-shaped because the geodesics in $\mathscr{U}$ correspond to the rays through the origin in $U$.
\end{proof}

\begin{lem}\label{lem:inj on vp}
$\exp_p$ is injective on $V_p$.
\end{lem}

\begin{proof}
Consider $\mathscr{U}_r=\exp_p(U_r)$ where $U_r = V_p\cap B_r(0)$. Note that every point in $\mathscr{U}_r$ can be reached from $p$ by a geodesic entirely in $\mathscr{U}$, since $U_r$ is star-shaped.

By Lemma \ref{lem:loc inj}, any failure of injectivity for $\exp_p$ must arise from a situation where $\exp_p(U_r)$ is not simply connected. This can only happen if two radial geodesics from $p$ meet after going around opposite sides of some cone point $\zeta$. This situation is depicted in Figure \ref{fig:inj}.

\begin{center}
\setlength{\unitlength}{.23mm}
\begin{picture}(660,200)(-220,-100)

\put(150,0){\circle*{6}}
\put(-50,0){\circle*{6}}

\put(-57,-17){$p$}
\put(143,-17){$\zeta$}

\multiput(-50,0)(26,0){8}{\line(1,0){20}}
\multiput(150,0)(20,20){5}{\line(1,1){16}}
\multiput(150,0)(20,-20){5}{\line(1,-1){16}}

\put(230,100){$l_1$}
\put(230,-110){$l_2$}

\qbezier(-100,-50)(-160,0)(-100,50)
\qbezier(-100,50)(70,180)(250,50)
\qbezier(-100,-50)(70,-180)(250,-50)
\qbezier(250,50)(320,0)(250,-20)
\qbezier(250,-50)(320,-0)(250,20)
\qbezier(250,-20)(210,-30)(200,0)
\qbezier(250,20)(210,30)(200,0)
\qbezier(200,0)(195,40)(150,0)
\qbezier(200,0)(195,-40)(150,0)

\qbezier(-50,0)(100,120)(260,-10)
\qbezier(-50,0)(100,-120)(260,10)

\put(60,64){$\gamma$}

\qbezier(-50,0)(150,0)(215,60)

\put(110,23){$\eta$}

\put(-110,80){$\mathscr{U}_r$}

\end{picture}

\begin{figure}[h]
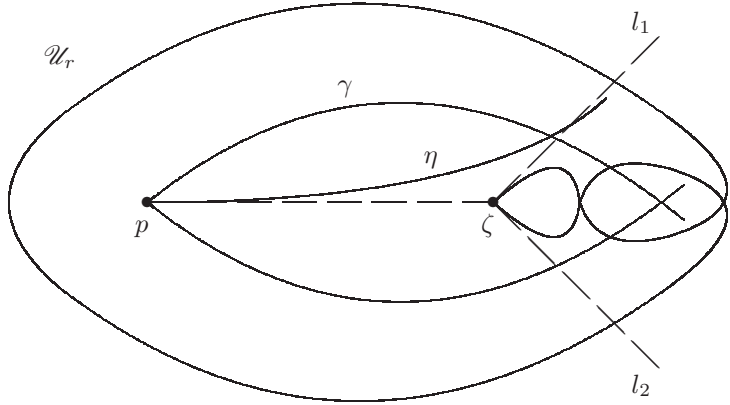

\caption{Potential non-injectivity of $\exp_p$ `behind' a cone point $\zeta$.}\label{fig:inj}
\end{figure}

\end{center}

In Fig \ref{fig:inj}, $l_1$ and $l_2$ are geodesics from $p$ through the cone point $\zeta$ making angle $\pi$ on either side of $\zeta$.  A pair of geodesics demonstrating potential non-injectivity of the exponential map `behind' $\zeta$ are shown; we see that $\exp_p(U_r)$ is not simply connected, due to the cone point. It is clear that one of the two intersecting geodesics, say $\gamma$, intersects $l_1$ or $l_2$, say $l_1$. Take a tangent vector in $V_p$ at $p$ very close to the vector between $p$ and $\zeta$ and on $\gamma$'s side of $\zeta$. The geodesic $\eta$ it generates must stay close to $l_1$ for a long time, and therefore must intersect $\gamma$ before it leaves $\mathscr{U}_r$. But then restricting $U_r$ to, say, those vectors generating the top half of Figure \ref{fig:inj} so that we have a simply connected region produces a contradiction to Lemma \ref{lem:loc inj}.

\end{proof}

\begin{prop}\label{prop:no two int}
The intersection of two distinct $\tilde g$-geodesics in $\tilde S$ for $g\in \mathscr{M}_{ncpc}(S)$ has at most one connected component.
\end{prop}

This has an immediate corollary:

\begin{cor}\label{cor:minimizing}
No $\tilde g$-geodesic self-intersects, and every $\tilde g$-geodesic is minimizing.
\end{cor}

\begin{proof}[Proof of Proposition \ref{prop:no two int}]
Suppose we have two geodesics $l_1$ and $l_2$ which intersect at $p$ and $q$, but not between. By Lemma \ref{lem:inj on vp}, at least one of these geodesics is singular and hits a cone point between $p$ and $q$. We will prove the result by induction on the number of cone points which lie between $l_1$ and $l_2$ between $p$ and $q$ or which lie on them but make angle $>\pi$ on the `between' side.

Suppose there is only one such cone point, $\zeta$ on $l_1$ or $l_2$ between their intersections making angle $>\pi$ on the `between' side. Assume $\zeta$ belongs to $l_1$. Let $l_1'$ be the geodesic through $p$ and $\zeta$ making angle $\pi$ on the side of $\zeta$ to which the segment of $l_2$ between $p$ and $q$ lies (see Figure \ref{fig:intersect}).

\begin{center}
\setlength{\unitlength}{.23mm}
\begin{picture}(660,120)(-180,-40)

\put(80,50){\circle*{6}}
\put(-50,0){\circle*{6}}
\put(209,50){\circle*{6}}

\put(-57,-17){$p$}
\put(75,60){$\zeta$}
\put(205,60){$q$}

\put(260,50){$l_1$}
\put(100,-30){$l_2$}
\put(240,-10){$l_1'$}

\qbezier(-50,0)(15,25)(80,50)
\qbezier(80,50)(90,50)(250,50)
\qbezier(80,50)(165,25)(230,0)
\qbezier(-50,0)(100,-60)(230,70)

\end{picture}

\begin{figure}[h]
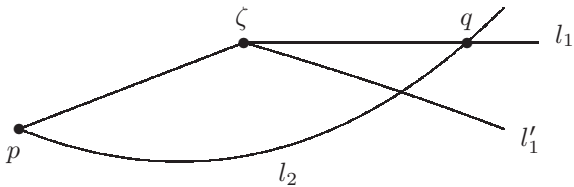

\caption{Base case for the induction in Propostion \ref{prop:no two int}.}\label{fig:intersect}
\end{figure}

\end{center}

We see then that $l_1'$ and $l_2$ intersect before $l_2$ reaches $q$. We know that $l_2$ is non-singular over this segment. Although $l_1'$ is not, we may approximate it by a non-singular segment on the $l_2$-side of $\zeta$ since the angles along $l_1$ on that side are always $\pi$. The intersection of this approximating segment and $l_2$ yields a contradiction to Lemma \ref{lem:inj on vp}.

Now suppose we have proven the result when $n$ or fewer cone points lie between $l_1$ and $l_2$. We can reduce to the case where all the cone points lie on the geodesics by re-drawing $l_1$ from $p$ through a cone point between $l_1$ and $l_2$. At the last cone point before $q$ on $l_1$ or $l_2$ where an angle $>\pi$ is made, replace the geodesic by the angle $\pi$ geodesic on the correct side of that cone point as above. We see that the number of cone points which are between $l_1$ and $l_2$ and at which one of the geodesics makes angle $>\pi$ strictly decreases. By the inductive hypothesis, we are done. 
\end{proof}

We will need the following below. An analogous result is the key to Fathi and Croke-Fathi-Feldman's extension to the `no conjugate points' assumption for $g_1$. They reference Morse (\cite{morse}) for the version without cone points; we give an alternate proof here allowing cone points.

\begin{prop}\label{prop:morse}
Let $g_1,g_2\in \mathscr{M}_{ncpc}(S)$ for a surface of genus $\geq 2$. Then there exists a constant $K>0$, depending only on $g_1$ and $g_2$, such that any $\tilde g_1$-geodesic contains a $\tilde g_2$-geodesic in its $K$-neighborhood, and vice versa.
\end{prop}

\begin{proof}
Let $\Gamma=\pi_1(S)$. Each of $(\tilde S, \tilde g_1)$ and $(\tilde S, \tilde g_2)$ are quasi-isomorphic to $\Gamma$ with a word metric, and hence are $\delta$-hyperbolic metric spaces for some $\delta$. There is a quasi-isometry $(\tilde S, \tilde g_1) \to (\tilde S, \tilde g_2)$. Since any $\tilde g_1$-geodesic is minimizing, it is a quasi-geodesic for $\tilde g_2$. The result follows from the stability of quasi-geodesics in $\delta$-hyperbolic spaces (see, e.g., \cite[III.H Theorem 1.7]{bh}).
\end{proof}

If $S$ is endowed with a metric $g$ in $\mathscr{M}_{ncpc}(S)$, then $(\tilde S, \tilde g)$ can be compactified by adding the boundary at infinity $\partial^\infty(\tilde S)$, which consists of equivalence classes of asymptotic geodesic rays. Note that since $\Gamma$ acts cocompactly on $\tilde S$, we can identify $\partial^\infty(\tilde S)$ with $\partial^\infty(\Gamma)$ where $\Gamma$ is endowed with the word metic for any finite generating set.

By Corollary \ref{cor:minimizing} all geodesics for $\tilde g\in\mathscr{M}_{ncpc}(S)$ are minimizing, so to any geodesic $l$ in $(\tilde S,\tilde g)$ we can associate $l(+\infty)$ and $l(-\infty)$, its forward and backward endpoints in $\partial^\infty(\tilde S)$. If $g$ has strictly negative curvature, any pair of distinct points $\{\xi_1, \xi_2\}$ in $\partial^\infty(\tilde S)$ specifies a unique geodesic in $(\tilde S,\tilde g)$. In nonpositive curvature, any two geodesics in the universal cover with the same endpoints in $\partial^\infty(\tilde S)$ bound a \emph{flat strip}, that is, an isometrically embedded copy of $\mathbb{R} \times [a,b]$ (\cite[Thm 4.1]{lgreen}). With only the no conjugate points assumption, they bound a (not necessarily flat) strip and through each (non-cone) point in this strip runs a unique geodesic with the same endpoints at infinity (\cite[Cor 3.1]{lgreen}). 

Let $\mathcal{G}_{\tilde g}$ be the set of images of $\tilde g$-geodesics. We endow $\mathcal{G}_{\tilde g}$ with the topology of convergence on compact sets. Let

\[ \mathscr{G}(\tilde S) = [(\partial^\infty(\tilde S) \times \partial^\infty(\tilde S))\setminus \Delta]/((x,y)\sim(y,x)) \]
(where $\Delta$ is the diagonal) be the set of unordered pairs of distinct points in the boundary of $\tilde S$.

\begin{defn}
Let $\partial_{\tilde g}: \mathcal{G}_{\tilde g} \to \mathscr{G}(\tilde S)$ send a geodesic $l$ to $\{l(+\infty),l(-\infty)\}$. 
\end{defn}
Note that $\mathcal{G}_{\tilde g}$ depends on the metric, while $\mathscr{G}(\tilde S)$ depends only on $\Gamma$. Due to Proposition \ref{prop:morse} (applied with $g_2$ of strictly negative curvature, for instance) $\partial_{\tilde g}$ is surjective. As noted above, it is injective away from geodesics contained in strips.

\begin{defn}
We say $\tilde g_1$-geodesic $l$ \emph{corresponds to} $\tilde g_2$-geodesic $l'$ if $\partial_{\tilde g_1}(l) = \partial_{\tilde g_2}(l').$
\end{defn}

\begin{defn}
Let $\mathcal{G}^\circ_{\tilde g}$ be the set of non-singular $\tilde g$-geodesics, i.e. those not hitting any cone point. Let $\mathcal{G}^*_{\tilde g} = \overline{\mathcal{G}^\circ_{\tilde g}}$.
\end{defn}

We now want to prove that $\mathcal{G}^\circ_{\tilde g}$ is non-empty, and to characterize what will turn out to be almost every geodesic in $\mathcal{G}^*_{\tilde g}$ (see Lemma \ref{lem:countable}). The first will be essential for our construction of the Liouville current (Section \ref{sec:liou}); the second is essential to use Bankovic-Leininger's characterization of the support of that current (Proposition \ref{prop:supp}).

\begin{prop}\label{prop:ae}
Let $g\in \mathscr{M}_{ncpc}(S)$. At any non-cone point $p\in \tilde S$ the set of geodesics through $p$ in $\mathcal{G}^\circ_{\tilde g}$ is full measure with respect to the angular measure.
\end{prop}

\begin{proof}
Two distinct vectors based at $p$ generate geodesics that never intersect elsewhere, by Proposition \ref{prop:no two int}, so each cone point lies on the geodesic generated by at most one such vector. There are only countably many cone points, so the proposition follows.
\end{proof}

\begin{prop}[cf. Prop 2.3, \cite{bl}]
Let $g\in \mathscr{M}_{ncpc}(S)$. If $l$ is a $\tilde g$-geodesic containing at most one cone point and making angle $\pi$ on one side at that point, then $l\in \mathcal{G}^*_{\tilde g}$.
\end{prop}

\begin{proof}
Let $\zeta$ be the cone point contained by $l$. Let $v^*$ be the tangent vector to $l$ at $\zeta$. For any sequence $(v_n)$ of vectors tangent to geodesics $(l_{v_n})$ in $\mathcal{G}^\circ_{\tilde g}$ with $v_n$ approaching $v^*$ via basepoints on the $\pi$-side of $l$, the forward ray of $l_{v_n}$ converges to the forward ray of $l$, since there are no cone points along that ray. The same is true for backward rays since there are again no cone points along those rays and the angle $l$ makes at that side of $\zeta$ is $\pi$. 

By Proposition \ref{prop:ae} almost every vector for the Lebesgue measure on $T^1(\tilde S-\tilde P)$ is tangent to a geodesic in $\mathcal{G}^\circ_{\tilde g}$. So a sequence answering the requirements of $(v_n)$ above exists.

\end{proof}



%

We close this section by proving a claim mentioned in the introduction -- that the volume in $T^1(S\setminus P)$ of flat strips for a metric in $\mathscr{M}_{npc}(S)$ is zero.

\begin{prop}\label{prop:flat strip volume}
Let $g$ be in $\mathscr{M}_{npc}(S)$ where $S$ has genus at least 2, and let $B$ be the set of all tangent vectors to flat strips in $T^1_g$. Then $vol_g(B)=0$ where $vol$ denotes the usual measure on $T^1_g$.
\end{prop}

\begin{proof}


Note that no flat strip may contain a cone point, so the argument necessarily takes place in $S\setminus P$.

Assume the contrary. Let $\mathcal{F}_g$ be the set of flat strips for $g$. We may assume that each strip in $F$ is maximal, in the sense that it cannot be extended to a flat strip with greater width. As $S$ has genus $>2$, such a maximum width must exist. The volume in $T^1(S\setminus P)$ of any single flat strip is zero, so if $B$ is to have positive volume, there must exist some $\delta>0$ such that there are infinitely many flat strips with width $>\delta$.

Let $(F_j)$ be a sequence of distinct (maximal) flat strips with width $>\delta$. There is a subsequence $(F_{i_j})$ converging to a maximal flat strip $G$ with width $>\delta$ in the sense that there is a tangent vector $v_G$ to the geodesic direction to the strip $G$ and a sequence of vectors $(v_{i_j})\to v_G$ such that $v_{i_j}$ is tangent to the geodesic direction at the center of the strip $F_{i_j}$. We may assume $G$ does not belong to the sequence $F_{i_j}$ by removing it if necessary. For sufficiently large $j$, $F_{i_j}$ and $G$ overlap, and so we can define the angle $\alpha_{i_j}$ between $v_{i_j}$ and $v_G$ using the fact that the strips are flat. We note that this angle cannot be zero, else $F_{i_j} \cup G$ would provide an extension of either $F_{i_j}$ or $G$, contradicting maximality.

Now consider the flat half-strips, i.e. isometric immersions of $\mathbb{R}_+ \times [a,b]$ which are obtained from restricting $G$ to $[L,\infty)$ in the geodesic coordinate, for any $L$. Again, each such flat half-strip is contained in a maximal width flat half-strip. Let $W_{max}$ be the supremum of all the widths of these. Since the genus of $S$ is at least 2, it is again clear that $W_{max}<\infty$. Pick some flat half-strip arising from $G$ with width at least $W_{max}-\frac{\delta}{16}$; call it $G^+$. Note that flat half-strips $F_{i_j}^+$ coming from the $F_{i_j}$ approach $G^+$, and that the angles between these flat half strips (for sufficiently large $j$) are $\alpha_{i_j}$.

Now it is possible to use an argument of Cao and Xavier in \cite{cao-xavier} (also discussed in \cite[Prop 3.1]{CS}) to produce from $G^+$ and the $F_{i_j}^+$ a flat half-strip along $G$ with width $\frac{\delta}{8}$ greater than the width of $G^+$, producing a contradiction. 

\end{proof}

\begin{rem}
The argument above also works under the weaker assumption that there are uncountably many maximal flat strips. The cone points play no role in the question -- in fact once you have infinitely many strips with width at least $\delta$, one could perform a surgery on $\frac{\delta}{3}$-neighborhoods of the cone points and produce a complete, nonpositively curved metric with infinitely many strips of width at least $\frac{\delta}{3}$ and proceed from there. Though Cao and Xavier's argument will not work under the weaker assumption of no conjugate points, the fact that we have not exploited the full strength of the positive volume assumption above gives some hope that the following question may have a positive answer.
\end{rem}

\begin{ques}
Let $g$ be a Riemannian metric on a closed surface of genus at least two without conjugate points. Does the set of all tangent vectors for strips have zero volume?
\end{ques}

%

\section{Geodesic currents}

Apart from the original work of Croke \cite{croke}, proofs of marked length spectrum rigidity for surfaces all rely on the fundamental work of Otal relating the marked length spectrum to geodesic currents. (See \cite{bonahon_ends, bonahon_geometry} for good references on geodesic currents.)

\begin{defn}
A \emph{geodesic current} on $(S,g)$ is a $\Gamma$-invariant Radon measure on $\mathscr{G}(\tilde S)$.
\end{defn}

Write $\mathscr{C}(S)$ for the set of geodesic currents on $S$. Again, this depends only on $\Gamma$, not the particular metric. We endow $\mathscr{C}(S)$ with the weak*-topology.

A first, and crucial, example is the following. Let $\gamma$ be a closed geodesic on $(S,g)$. The set of all lifts of $\gamma$ to $(\tilde S,\tilde g)$ maps under $\partial_{\tilde g}$ to a discrete, $\Gamma$-invariant set of points in $\mathscr{G}(\tilde S)$. Denote by $\langle \gamma \rangle$ the current given by the counting measure on this set.

We now record a few fundamental facts about geodesic currents. The first three theorems can be found in the work of Bonahon.

\begin{thm}[See \cite{bonahon_ends}, Proposition 4.2]
The set of real multiples of geodesic currents of the form $\langle \gamma \rangle$ for $\gamma\in\Gamma$ is dense in $\mathscr{C}(S)$.
\end{thm}

\begin{thm}[See \cite{bonahon_ends}, \S III and \S IV]
There is a continuous, symmetric, bilinear form, called the \emph{intersection number}

\[ i: \mathscr{C}(S)\times \mathscr{C}(S) \to \mathbb{R} \]
which extends the geometric intersection number in the sense that for any two primitive elements $\gamma_1, \gamma_2\in \Gamma$, $i(\langle\gamma_1\rangle,\langle\gamma_2\rangle)$ is the number of intersections in $S$ of their geodesic representatives.
\end{thm}

Note that the intersection of two geodesics in $\mathcal{G}_{\tilde g}$ can be detected by the ordering of their endpoints around the boundary at infinity. In particular, intersection is independent of the choice of metric.

\begin{thm}[Implicit in \cite{bonahon_ends}, \S4.2]\label{thm:intersection length}
For $\gamma \in \Gamma$, let $\mathcal{I}(\tilde \gamma)$ be the subset of $\mathcal{G}_{\tilde g}$ consisting of all $\tilde g$-geodesics intersecting $\tilde \gamma$, the axis for $\gamma$. Let $I(\tilde \gamma)$ be any fundamental domain for the action of $\gamma$ on $\mathcal{I}(\tilde \gamma)$. Then for any $\mu \in \mathscr{C}(S)$

\[ i(\mu,\langle\gamma\rangle) = \mu(\partial_{\tilde g}(I(\tilde \gamma))). \]
\end{thm}

The following theorem of Otal is fundamental to the entire approach to MLS rigidity through currents.

\begin{thm}[\cite{otal}, Th\'eor\`eme 2]\label{thm:lengths determine liou}
For $\mu_1,\mu_2 \in \mathscr{C}(S)$, $\mu_1=\mu_2$ if and only if $i(\mu_1,\langle\gamma\rangle) = i(\mu_2,\langle\gamma\rangle)$ for all $\gamma\in \Gamma$. 
\end{thm}

%

\section{The Liouville current}\label{sec:liou}

We follow Hersonsky-Paulin and Bankovic-Leininger in extending the Liouville current in the presence of cone points. It is defined as follows.

Let $k$ be an oriented geodesic segment for $\tilde g$ not hitting any cone points. Let $\mathcal{G}^\circ_{\tilde g}(k)\subset \mathcal{G}_{\tilde g}$ be the set of $\tilde g$ geodesics transversally intersecting the interior of $k$ which do not hit any cone point. We introduce coordinates $(t,\theta)$ on $\mathcal{G}^\circ_{\tilde g}(k)$ by letting $t(l)$ be the distance from the starting point of $k$ to its intersection with $l\in\mathcal{G}^\circ_{\tilde g}(k)$ and by letting $\theta(l)\in(0,\pi)$ be the angle between the forward direction of $k$ and $l$. Let $D^\circ(k)\subset [a,b]\times (0,\pi)$ be all the coordinate pairs arising this way. As there are only countably many cone points in $(\tilde S,\tilde g)$, $D^\circ(k)$ is a full Lebesgue measure subset of $[a,b]\times (0,\pi)$ (using Proposition \ref{prop:ae}). The bijection $D^\circ(k) \to \mathcal{G}_{\tilde g}^\circ(k)$ allows us to push forward the measure defined on $D^\circ(k)$ by

\[ \hat L_{g} = \frac{1}{2}\sin\theta d\theta dt \]
to a measure $\hat L_g(k)$ on $\mathcal{G}^\circ_{\tilde g}(k)$.

Cover $\mathcal{G}^\circ_{\tilde g}$, the set of $\tilde g$-geodesics not hitting any cone point, by sets of the form $\mathcal{G}^\circ_{\tilde g}(k)$. Let $G^\circ_{\tilde g}$ be the set of parametrized geodesics in $(\tilde S,\tilde g)$. We then have

\begin{prop}
The measures $\hat L_g(k)$ are transverse invariant measures on the foliation of $G^\circ_{\tilde g}$ by geodesics, and therefore define a global measure on $\mathcal{G}^\circ_{\tilde g}$, denoted by $\hat L_g$. This measure is $\Gamma$-invariant.
\end{prop}

\begin{proof}
$\Gamma$-invariance follows from the definition of $\hat L_g$, provided the other parts of the proposition can be proven.

That the $\hat L_g(k)$ are transverse invariant measures on the foliation by non-singular geodesics follows from an argument pointed out by Hersonsky and Paulin \cite[Prop 4.11]{hp}: applying a countable cutting procedure to the well-known analogous result without singularities.
\end{proof}

\begin{defn}
The \emph{Liouville current} $L_g$ is the element of $\mathscr{C}(S)$ formed by extending $\hat L_g$ to all of $\mathcal{G}_{\tilde g}$ by setting $\hat L_g(\mathcal{G}_{\tilde g} \setminus \mathcal{G}^\circ_{\tilde g})=0$ and pushing forward $\hat L_g$ by $\partial_{\tilde g}$ to a measure on $\mathscr{G}(\tilde S)$.
\end{defn}

The following proposition gives the key relationship between this current and the geometry of $(S,g)$:

\begin{prop}\label{prop:length}
For any geodesic segment $k$ and $\mathcal{I}(k)$ the set of all $\tilde g$-geodesics intersecting $k$ transversally in its interior, $L_g(\partial_{\tilde g}(\mathcal{I}(k))) = length_g(k)$.
\end{prop}

\begin{proof}
This follows easily from integrating the local coordinate expression for $\hat L_g$.
\end{proof}

\begin{prop}\label{prop:int gamma}
For every $\gamma \in \Gamma$, 

\[ i(L_g,\langle\gamma\rangle) = length_g(\gamma). \]
\end{prop}

\begin{proof}
Every $\gamma$ is realized by a finite union of geodesic segments. Then use Theorem \ref{thm:intersection length} and Proposition \ref{prop:length}.
\end{proof}

Combining these results, we get the following important result linking the marked length spectrum to the Liouville currents.

\begin{prop}\label{prop:preserves Liou}
Suppose that $g_1, g_2 \in \mathscr{M}_{ncpc}(S)$ have the same marked length spectrum. Then $L_{g_1} = L_{g_2}$.
\end{prop}

\begin{proof}
By Theorem \ref{thm:lengths determine liou} we need only check that $i(L_{g_1},\langle\gamma\rangle) = i(L_{g_2},\langle\gamma\rangle)$ for all $\gamma \in \Gamma$. But this is provided by Proposition \ref{prop:int gamma} and equality of the marked length spectrum.
\end{proof}

%

\section{Cone points}\label{sec:cone}

In \cite{bl}, Bankovic and Leininger provide a novel way to detect cone points using the Liouville current. They work in the setting of nonpositively curved Euclidean cone metrics -- metrics which are locally Euclidean away from a finite set of cone points which have cone angles $>2\pi$. The reader can easily use the results of Section \ref{sec:geodesics} verify that the results in the first four sections of their paper hold for metrics in $\mathscr{M}_{ncpc}(S)$. The key steps are to replace their use of the $\mathrm{CAT}(0)$ assumption in a few places by results we proved in Section \ref{sec:geodesics}, and to replace references to `flat strips' by `strips,' using the work in \cite{lgreen} noted above.

Their identification of cone points relies on carefully studying the support of $L_g$. We restate their main results here.

\begin{defn}
Let $\mathcal{G}^2_{\tilde g} \subset \mathcal{G}^*_{\tilde g}$ be the set of geodesics in $\mathcal{G}^*_{\tilde g}$ containing at least two cone points.
\end{defn}

\begin{lem}[See \cite{bl}, Corollary 2.5]\label{lem:countable}
$\mathcal{G}^2_{\tilde g}$ is countable.
\end{lem}

\begin{proof}
Bankovic and Leininger argue that containing more than one cone point and yet still lying in $\mathcal{G}^*_{\tilde g}$, that is, being the limit of non-singular geodesics, forces a geodesic to make angle $\pi$ on one side at every cone point with the side on which this angle lies switching at most once. As there are countably many cone points, the result follows.
\end{proof}

From its definition, the following is clear:

\begin{prop}\label{prop:supp}
$Supp(L_g) = \partial_{\tilde g}(\mathcal{G}^*_{\tilde g})$.
\end{prop}

Then by examining the support of $L_g$, specifically special sequences of points in $\partial_{\tilde g}(\mathcal{G}^*_{\tilde g})$, they prove the following theorem:

\begin{thm}[\cite{bl}, see Prop 4.5 and work in section 4]\label{thm:cone geodesics}
Let $g_1,g_2\in \mathscr{M}_{npc}(S)$ have the same marked length spectrum. Then there is a $\Gamma$-equivariant, bijective isometry $F_c:cone(\tilde g_1) \to cone(\tilde g_2)$ between the cone points of $(\tilde S,\tilde g_1)$ and those of $(\tilde S,\tilde g_2)$. Furthermore, this isometry is induced by a map $\mathcal{F}_c: \mathcal{G}^*_{\tilde g_1}(\zeta)\setminus\Omega\to\mathcal{G}_{\tilde g_2}$, for $\Omega$ countable, which takes the set of geodesics in $\hat{\mathcal{G}}^*_{\tilde g_1}\setminus \Omega$ passing through a cone point $\zeta$ of $\tilde g_1$ to geodesics passing through the cone point $F_c(\zeta)$ of $\tilde g_2$ in such a way that $\mathcal{F}_c(l)$ corresponds to $l$.
\end{thm}

\begin{proof}[Outline of proof]
Bankovic and Leininger introduce the notion of $(L_g,\Omega)$\emph{-chains} (\cite[\S4]{bl}). An $(L_g,\Omega)$-chain is a sequence $(x_i)$ of points in $\partial^\infty(\tilde S)$ such that for all $i$,

\begin{itemize}
	\item $\{x_i,x_{i+1}\} \in supp(L_g)\setminus \Omega$, where $\Omega$ is countable, and
	\item $x_i, x_{i+1}, x_{i+2}$ is counterclockwise ordered and the set of all points in $supp(L_g)$ between $\{x_i,x_{i+1}\}$ and $\{x_{i+1},x_{i+2}\}$ (in the sense that a geodesic realizing this point is between geodesics realizing $\{x_i, x_{i+1}\}$ and $\{x_{i+1},x_{i+2}\}$) is only $\{ \{x_i,x_{i+1}\},\{x_{i+1},x_{i+2}\}\}.$ (See \cite[\S2.3]{bl} for the precise definition of `between.')
\end{itemize}
Figure \ref{fig:chain} illustrates a few steps in an $(L_g,\Omega)$-chain.

\begin{center}
\setlength{\unitlength}{.18mm}
\begin{picture}(660,200)(-220,-80)

\qbezier(40,0)(40,80)(120,80)
\qbezier(200,0)(200,80)(120,80)
\qbezier(40,0)(40,-80)(120,-80)
\qbezier(200,0)(200,-80)(120,-80)

\put(120,0){\circle*{6}}
\qbezier(120,80)(120,40)(120,0)
\qbezier(120,0)(80,-20)(45,-35)
\qbezier(120,0)(130,-20)(158,-75)
\qbezier(120,0)(190,30)(195,33)
\qbezier(120,0)(100,36)(80,72)

\put(120,85){$x_1$}
\put(25,-40){$x_2$}
\put(160,-85){$x_3$}
\put(198,35){$x_4$}
\put(70,82){$x_5$}

\put(110,-20){$\zeta$}

\end{picture}

\begin{figure}[h]
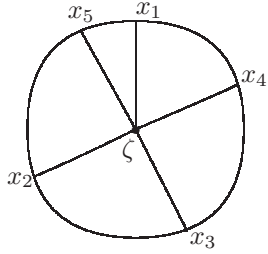

\caption{A few steps in an $(L_g,\Omega)$-chain. The angle between $[\zeta,x_i]$ and $[\zeta,x_{i+1}]$ is $\pi$.}\label{fig:chain}
\end{figure}

\end{center}

Let $\zeta$ be a cone point for $\tilde g$. Then there is a countable set $\Omega$ containing $\partial_{\infty}(\mathcal{G}^2_{\tilde g_i})$ for $i=1,2$ (thanks to Lemma \ref{lem:countable}), and a set of $(L_{g},\Omega)$-chains corresponding to geodesics passing through $\zeta$. In addition, $\zeta$ is uniquely specified by these chains in that given any $(L_{g},\Omega)$-chain, it corresponds to a sequence of $\tilde g$-geodesics through a unique cone point (\cite[Prop 4.1]{bl}). Furthermore, one can detect from conditions on $\partial^\infty(\tilde S)$ alone whether two $(L_g,\Omega)$-chains specify the same cone point (\cite[Lemma 4.4]{bl}). As the marked length spectrum determines $L_g$, and hence $supp(L_g)$, equality of marked length spectra for $g_1$ and $g_2$ implies that we have a bijective map $cone(\tilde g_1) \to cone(\tilde g_2)$ (see \cite[proof of Theorem 5.1]{bl}).

We now describe the map $\mathcal{F}_c$. Let $\zeta$ be a cone point for $\tilde g_1$. It is the unique cone point specified by an $(L_{g_1},\Omega)$-chain (see \cite[Lemma 4.2]{bl}). The $g_1$-geodesics realizing this chain correspond to the $\tilde g_2$-geodesics doing the same. As noted above, these $\tilde g_2$-geodesics pass through $F_c(p)$. This proves the last statement of the Theorem.

Finally, we want to show that $F_c$ is an isometry. Let $\zeta_1, \zeta_2 \in cone({\tilde g_1})$, and let $k$ be the $\tilde g_1$-geodesic segment between them. Consider the set $\mathcal{G}_{\tilde g_1}(k)$ of all $\tilde g_1$-geodesics $l$ intersecting $k$ transversally and such that $\partial_{\tilde g_1}(l) \notin \Omega$. Using the fact that $F_c$ is induced by the map $\mathcal{F}_c$ described above and the fact that cone points cannot lie in the interior of a flat strip, it is not hard to see that the $\tilde g_2$-geodesics with the same endpoints as those in $\mathcal{G}_{\tilde g_1}(k)$ are precisely those intersecting the $\tilde g_2$ geodesic segment $k'$ between $F_c(\zeta_1)$ and $F_c(\zeta_2)$ and not having endpoints in $\Omega$. As $\Omega$ is countable, the set of geodesics with endpoints in $\Omega$ is measure zero for any Liouville current. Then using Propositions \ref{prop:length} and \ref{prop:preserves Liou},

\begin{align}
	d_{\tilde g_1}(\zeta_1, \zeta_2) & = L_{g_1}(\partial_{\tilde g_1}(\mathcal{I}(k)))  = L_{g_1}(\partial_{\tilde g_1}(\mathcal{G}_{\tilde g_1}(k))) \nonumber \\
						& = L_{g_2}(\partial_{\tilde g_2}(\mathcal{G}_{\tilde g_2}(k')))  = L_{g_2}(\partial_{\tilde g_2}(\mathcal{I}(k'))) \nonumber \\
						& = d_{\tilde g_2}(F_c(\zeta_1),F_c(\zeta_2)). \nonumber
\end{align}
This completes the proof.

\end{proof}

\begin{rem}
As we will see below, the work of Fathi and Croke-Fathi-Feldman (using Otal's ideas) produces a similar isometry between points where the curvature is negative. Since cone points with angle $>2\pi$ are like point masses of negative curvature (as in the Gauss-Bonnet theorem with cone points), it is not surprising that this result would hold. What is particularly nice about Bankovic and Leininger's approach is its comparatively `low-tech' approach -- they only need to know the support of the Liouville current.
\end{rem}

Let $T^1_g$ be the subset of $T^1(\tilde S-\tilde P)$ consisting of unit tangent vectors for non-singular $\tilde g$-geodesics. Note that almost every vector in $T^1(\tilde S-\tilde P)$ (with respect to the usual volume) is in $T^1_g$.

\begin{prop}\label{prop:non-sing}
There is a volume zero, geodesic flow-invariant subset $B$ of $T^1_{g_1}$ such that the following is true. Let $l$ be a non-singular $\tilde g_1$-geodesic whose tangent vectors do not lie in $B$. Then no $\tilde g_2$-geodesic corresponding to $l$ is singular. Specifically, $B$ is the set of tangent vectors for non-singular geodesics which do not belong to a strip bordering a cone point.
\end{prop}

\begin{proof}
From Theorem \ref{thm:cone geodesics} we know that any cone point for $\tilde g_2$ is $F_c(\zeta)$ for $\zeta$ some cone point of $\tilde g_1$. Under our assumption, $\zeta$ does not belong to $l$.

Let $r_1$ and $r_2$ be the geodesic rays from $\zeta$ to the endpoints $\partial_{\tilde g_1}(l) = \{x,y\}$. If $r_1$ and $r_2$ make angle $<\pi$ at $\zeta$, then it is not hard to see that there is an $(L_{g_1},\Omega)$-chain through $\zeta$ containing the endpoints of a geodesic $\gamma$ through $\zeta$ which makes angle $\pi$ on the $l$ side of $\zeta$, and with endpoints distinct from $x$ and $y$ and on the $\zeta$ side of $x$ and $y$. (See Figure \ref{fig:non-cone}.) The corresponding geodesic passing through $F_c(\zeta)$ for the $(L_{g_2},\Omega)$-chain has the same endpoints and also makes angle $\pi$ on the side of $F_c(\zeta)$ to which $x$ and $y$ lie.  We then see that any $\tilde g_2$-geodesic connecting $x$ and $y$ cannot pass through $F_c(\zeta)$, as desired.

\begin{center}
\setlength{\unitlength}{.23mm}
\begin{picture}(660,180)(-160,-80)

\qbezier(40,0)(40,80)(120,80)
\qbezier(200,0)(200,80)(120,80)
\qbezier(40,0)(40,-80)(120,-80)
\qbezier(200,0)(200,-80)(120,-80)

\put(140,0){\circle*{6}}
\put(150,0){$\zeta$}

\qbezier(80,-73)(100,0)(80,73)
\qbezier(80,-73)(110,-36)(140,0)
\qbezier(80,73)(110,36)(140,0)

\put(75,0){$l$}
\put(105,20){$r_1$}
\put(105,-17){$r_2$}

\qbezier(120,80)(160,0)(120,-80)

\put(140,30){$\gamma$}

\put(75,80){$x$}
\put(75,-85){$y$}

\qbezier(140,10)(128,0)(140,-10)

\end{picture}

\begin{figure}[h]
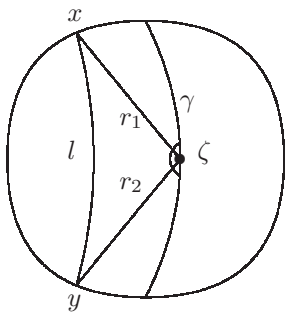

\caption{An $(L_{g_1},\Omega)$-chain geodesic $\gamma$ through $\zeta$ entirely to one side of $l$. The marked angle is $\pi$. There are uncountably many such $\gamma$ that can be drawn, so we can find one that avoids $\Omega$.}\label{fig:non-cone}
\end{figure}

\end{center}

Now suppose that $r_1$ and $r_2$ make angle $\pi$ on the $l$-side of $\zeta$. Then $l$ and $r_1\cup r_2$ bound a strip. Letting $B$ be the set of tangent vectors to non-singular geodesics belonging to strips bordering a cone point gives the result, using either Proposition \ref{prop:flat strip volume} if $g_1\in \mathscr{M}_{npc}(S)$ or the assumptions of Theorem \ref{main thm} if not.
\end{proof}

\begin{defn}
Let $\hat{\mathcal{G}}^\circ_{\tilde g_1} \subset \mathcal{G}^\circ_{\tilde g_1}$ be the set of all non-singular $\tilde g_1$-geodesics which do not belong to a strip bordering a cone point (as in Proposition \ref{prop:non-sing}). Let $\hat T^1_{g_1}$ be the set of all tangent vectors to geodesics in $\hat{\mathcal{G}}^\circ_{\tilde g_1}$.
\end{defn}


Proposition \ref{prop:non-sing} shows that no geodesic corresponding to $l$ is singular, so in fact no geodesic corresponding to $l$ can lie in a strip bordering a cone point. We thus have

\begin{cor}\label{cor:good geodesics}
Geodesics in $\hat{\mathcal{G}}^\circ_{\tilde g_1}$ correspond only to geodesics in $\hat{\mathcal{G}}^\circ_{\tilde g_2}$ and vice versa.
\end{cor}

%

\section{Angle correspondence}\label{sec:angles}

The following result is the key to proving a version of Theorem \ref{thm:cone geodesics} for points of negative curvature.

\begin{prop}\label{prop:angles}
There is a full measure set of $\theta\in(0,\pi)$ such that the following holds. Suppose that $l_1, l_2$ are $\tilde g_1$-geodesics in $\hat{\mathcal{G}}^\circ_{\tilde g_1}$ intersecting with angle $\theta$. Then any two corresponding $\tilde g_2$-geodesics $l'_1, l'_2$ intersect with angle $\theta$.
\end{prop}

\begin{rem}
Note that the angle between $l_1'$ and $l_2'$ is well-defined. Indeed, by definition of $\hat{\mathcal{G}}^\circ_{\tilde g_1}$ and Proposition \ref{prop:non-sing}, the $l_i'$ are non-singular. Hence they intersect in a well-defined angle, and the problematic configurations in Figure \ref{fig:transversal} do not occur. Further, if $l_1'$ and $l_1''$ both correspond to $l_i$, then they bound a flat strip. The fact that this strip is flat ensures that the angle between $l_1'$ and $l_2'$ is the same as the angle between $l_1''$ and $l_2'$.
\end{rem}

\begin{center}
\setlength{\unitlength}{.18mm}
\begin{picture}(660,200)(-220,-100)
	
\qbezier(-220,0)(-220,80)(-140,80)
\qbezier(-220,0)(-220,-80)(-140,-80)
\qbezier(-60,0)(-60,80)(-140,80)
\qbezier(-60,0)(-60,-80)(-140,-80)

\qbezier(-218,-20)(-180,5)(-60,10)
\qbezier(-218,20)(-180,0)(-63,-20)

\put(-200,17){$l_1$}
\put(-200,-30){$l_2$}

\put(-40,0){\vector(1,0){60}}
\put(-45,10){`corresp.'}

\qbezier(40,0)(40,80)(120,80)
\qbezier(200,0)(200,80)(120,80)
\qbezier(40,0)(40,-80)(120,-80)
\qbezier(200,0)(200,-80)(120,-80)

\qbezier(42,-20)(55,-5)(95,0)
\qbezier(95,0)(100,0)(130,0)
\qbezier(45,30)(55,20)(95,0)

\put(95,0){\circle*{6}}
\put(130,0){\circle*{6}}

\qbezier(130,0)(150,5)(193,40)
\qbezier(130,0)(150,-15)(180,-60)

\put(60,22){$l_1'$}
\put(60,-24){$l_2'$}

\put(240,10){or}

\qbezier(300,0)(300,80)(380,80)
\qbezier(460,0)(460,80)(380,80)
\qbezier(300,0)(300,-80)(380,-80)
\qbezier(460,0)(460,-80)(380,-80)

\put(380,0){\circle*{6}}

\qbezier(300,-10)(340,-5)(380,0)
\qbezier(300,10)(340,5)(380,0)
\qbezier(450,-45)(420,-15)(380,0)
\qbezier(450,45)(420,15)(380,0)

\put(320,14){$l_1'$}
\put(320,-25){$l_2'$}

\end{picture}

\begin{figure}[h]
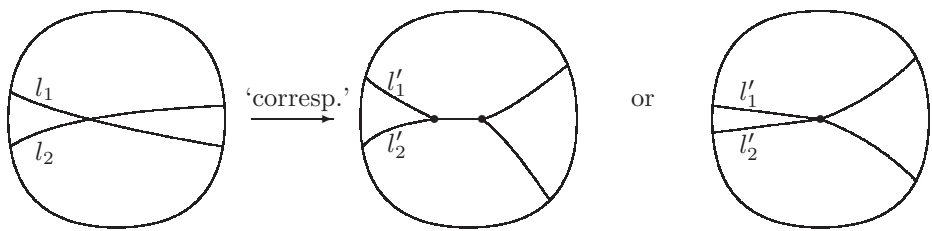

\caption{Transversal geodesics in $\hat{\mathcal{G}}_{\tilde g_1}^\circ$ and two pictures which do not occur for corresponding $\tilde g_2$-geodesics.}\label{fig:transversal}
\end{figure}

\end{center}

Proposition \ref{prop:angles} is implicit in the work of Otal, and has been reproven by Fathi and Hersonsky-Paulin (implicitly), and Croke-Fathi-Feldman (explicitly, see \cite[Lemma 1.5]{cff}). Each note that Otal's proof works in their situation with minor adjustments. The most significant adjustment comes in the case of cone singularities, and lies at the beginning of the argument. We describe this adjustment carefully here, then follow Otal and the subsequent presentations for the rest of the proof.

Let $\theta\in (0,\pi)$. Let $\theta\cdot v$ denote the rotation of $v$ by $\theta$ (using the fact that $S$ is oriented, this action commutes with that of $\Gamma$). Since $\hat T^1_{g_1}$ is full volume in $T^1_{g_1}$ and $\tilde P$ is countable, the set of $(v,\theta)$ such that both $v$ and $\theta \cdot v$ lie in $\hat T^1_g$ is full measure in $T^1_{g_1} \times (0,\pi)$ using the volume on the tangent bundle and the Lebesgue measure on $(0,\pi)$.

Let $(v,\theta) \in \hat T^1_{g_1}\times (0,\pi)$ be a pair such that $\theta\cdot v\in \hat T^1_{g_1}$. Consider the $\tilde g_1$-geodesics $l_v$ and $l_{\theta\cdot v}$ they generate in $\hat{\mathcal{G}}^\circ_{\tilde g_1}$. Under the correspondence with $\tilde g_2$-geodesics, these correspond to geodesics $l'_v$ and $l'_{\theta\cdot v}$ in $\hat{\mathcal{G}}^\circ_{\tilde g_2}$ (Corollary \ref{cor:good geodesics}). Let $\theta'(v,\theta)$ be the $\tilde g_2$-angle between $l'_v$ and $l'_{\theta\cdot v}$. Set $\theta'(v,0)=0$ and $\theta'(v,\pi)=\pi$.
By $\Gamma$-equivariance of the above construction, $\theta'(v,\theta)$ induces a map, also denoted $\theta'(v,\theta)$, almost everywhere on $\hat T^1_{g_1}(S)\times [0,\pi]$. The measure zero set where $\theta'(v,\theta)$ is undefined does not make any difference to the rest of Otal's argument, which is measure-theoretic from this point.

Note that the work of Bankovic and Leininger is crucial (via Proposition \ref{prop:non-sing}) to defining $\theta'(v,\theta)$ almost everywhere in the presence of cone points.  If we try to extend the definition to the problem cases shown in Figure \ref{fig:transversal} by, say, chosing one of the two natural `candidates' for the angle between $l'_v$ and $l'_{\theta \cdot v}$ then problems will arise when trying to prove Lemma \ref{lem:properties} below. Hersonsky and Paulin also have to address the well-definedness problem in their use of Otal's work. They provide a separate proof of the fact that the correspondence takes non-singular geodesics to non-singular geodesics (\cite[Lemma 4.15]{hp}) that uses their extensive work on the M\"obius current (also in \cite{hp}). The M\"obius current is defined using the crossratio and involves a fair bit of $\mathrm{CAT}(-1)$ technology, so it is not at all clear that something analogous can be done in the nonpositive curvature setting.

We now summarize the rest of Otal's argument to prove Proposition \ref{prop:angles}. For details of the proofs which are unaffected by our more general class of metrics, we refer to Otal's paper (\cite{otal}, particularly \S2 and \S3 through the proof of Lemme 8).

Let $\mu_g$ be the usual volume form on $T^1_g(S)$. Let $Vol(T^1_g(S))$ be the volume of $T^1_g(S)$ with respect to this measure.

\begin{defn}
Let 

\[ \Theta'(\theta)  = \frac{1}{Vol(T^1_g(S))} \int_{T^1_g(S)} \theta'(v,\theta) d\mu_g. \]
Since $\theta'(v,\theta)$ is defined for almost all $(v,\theta)$, this integral is valid for almost all $\theta$.
\end{defn}

The following properties are fairly straightforward. The third assertion follows from the Gauss-Bonnet theorem with singularities.

\begin{lem}\label{lem:properties}
Where defined, $\Theta'$ has the following properties:

\begin{itemize}
	\item[(1)] $\Theta':[0,\pi]\to[0,\pi]$ is increasing.
	\item[(2)] For all $\theta\in [0,\pi]$,
			\[ \Theta'(\pi-\theta) = \pi-\Theta'(\theta). \]
	\item[(3)] For all $\theta_1, \theta_2\in[0,\pi]$ such that $\theta_1+\theta_2 \in[0,\pi]$,
			\[ \Theta'(\theta_1+\theta_2) \geq \Theta'(\theta_1) + \Theta'(\theta_2). \]
\end{itemize}
\end{lem}

\begin{proof}
See \cite[Prop. 6]{otal}, or \cite[Prop. 4.16]{hp}. For Otal, $\Theta'$ is a homeomorphism, but Hersonsky-Paulin note that only measurability is needed.
\end{proof}

By the Jensen inequality, for any real-valued, strictly convex function $F$ on $[0,\pi]$, for every $\theta$,

\begin{equation}
 	F(\Theta'(\theta)) \leq \frac{1}{Vol(T^1_g(S))} \int_{T^1_g(S)} F(\theta'(v,\theta)) d\mu_g \label{Jensen}
\end{equation}
with equality for all $F$ if and only if $v \mapsto \theta'(v,\theta)$ is constant. Integrating with respect to $\sin\theta d\theta$ and applying Fubini, we have

\begin{equation}
	\int_0^\pi F(\Theta'(\theta)) \sin \theta d\theta \leq \frac{1}{Vol(T^1_g(S))} \int_{T^1_m(S)}\Big( \int_0^\pi F(\theta'(v,\theta) \sin\theta d\theta \Big) d\mu_g. \label{eqn:Jen int}
\end{equation}

\begin{prop}
For any convex function $F$,

\[\int_0^\pi F(\Theta'(\theta)) \sin\theta d\theta \leq \int_0^\pi F(\theta) \sin\theta d\theta \]
\end{prop}

\begin{proof}
Otal's proof of this result (\cite[Prop. 7]{otal}) can be extended to the `measurable' case as noted in Hersonsky-Paulin Prop. 4.17. We note that in the course of this proof, Otal proves that

\begin{equation}
	\frac{1}{Vol(T^1_g(S))} \int_{T^1_m(S)}\Big( \int_0^\pi F(\theta'(v,\theta) \sin\theta d\theta \Big) d\mu_g = \int_0^\pi F(\theta) \sin\theta d\theta. \label{eqn:otal}
\end{equation}
\end{proof}

\begin{lem}\label{lem:convex}
Suppose that $H$ is a measurable, increasing function on $[0,\pi]$, which is superraditive and commutes with the symmetry with respect to $\pi/2$. Suppose in addition that for any convex function $F$ on $[0,\pi]$

\[ \int_0^\pi F(H(\theta)) \sin\theta d\theta \leq \int_0^\pi F(\theta) \sin\theta d\theta. \]
Then $H=Id$.
\end{lem}

\begin{proof}
See \cite[Lemme 8]{otal} or \cite[Lemma 4.18]{hp}.
\end{proof}

\begin{proof}[Proof of Proposition \ref{prop:angles}]

\noindent With equation \eqref{eqn:Jen int} and \eqref{eqn:otal} we have that 

\[ \int_0^\pi F(\Theta'(\theta)) \sin \theta d\theta \leq \int_0^\pi F(\theta) \sin\theta d\theta. \]

\noindent By Lemma \ref{lem:convex}, $\Theta'=Id$ so we have equality in the equation above. But this equation comes from integrating equation \eqref{Jensen} over $[0,\pi]$ with respect to $\sin\theta d\theta$. So we must indeed have equality in equation \eqref{Jensen} for almost all $\theta$. As noted, this implies that $v \to \theta'(v,\theta)$ is constant for this full measure set of $\theta$.

Returning to the definition of $\Theta'$ and using $\Theta'=Id$, we have for such $\theta$ that

\[ \theta = \frac{1}{Vol(T^1_g(S))} \int_{T^1_g(S)} \theta'(v,\theta) d\mu_g. \]
As $\theta'(v,\theta)$ is constant in $v$, we must have $\theta'(v,\theta) = \theta$. This completes the proof.
\end{proof}

\begin{cor}\label{cor:neg now}
$g_1 \in \mathscr{M}_{npc}(S)$.
\end{cor}

\begin{proof}
This follows from the Gauss-Bonnet theorem and the angle correspondence given by Proposition \ref{prop:angles}. Indeed, any positive curvature for $\tilde g_1$ would be witnessed by some geodesic triangle with total angle sum $>\pi$, which cannot hold for the corresponding triangle under the nonpositively curved metric $\tilde g_2$. 
\end{proof}

%

\section{Isometry on points of negative curvature}\label{sec:neg}

We now prove a version of Theorem \ref{thm:cone geodesics} on the sets of strict negative curvature in $(S,g_1)$ and $(S,g_2)$. To do so, we follow ideas of Croke-Fathi-Feldman.

\begin{defn}
Define a partial relation on $\tilde S$ by $p\sim p'$ if almost every non-singular $\tilde g_2$-geodesic through $p'$ (with respect to the angular Lebesgue measure on the fiber of $T^1_{\tilde g_2}$ over $p'$) is a bounded distance from non-singular $\tilde g_1$ geodesic through $p$.
\end{defn}

Following Croke-Fathi-Feldman we have the following results.



\begin{lem}[\cite{cff}, Lemma 2.3]
The relation $\sim$ is the graph of a bijective function $F_{neg}$ between its domain $\tilde D_1$ and range $\tilde D_2$. $\tilde D_1$ and $\tilde D_2$ are $\Gamma$-invariant and $F_{neg}$ is $\Gamma$-equivariant.
\end{lem}

\begin{proof}
We need to prove the following. If $p\sim p_1'$ and $p\sim p_2'$, then $p_1'=p_2'$. That $F_{neg}$ is a bijection follows from reversing the roles of $\tilde g_1$ and $\tilde g_2$. (Recall that by Corollary \ref{cor:neg now} both metrics are now in $\mathscr{M}_{npc}(S)$.)

We are given that almost every non-singular $\tilde g_2$-geodesic through $p_i'$ is a bounded distance from a non-singular $\tilde g_1$-geodesic through $p$. We note that two distinct $\tilde g_2$-geodesics through $p_i'$ cannot correspond to the same $\tilde g_1$-geodesic through $p$ as they have different endpoints in $\partial_\infty(\tilde S)$. We then see that the map from $T^1_{\tilde g_2}(p_i')$ to $T^1_{\tilde g_2}(p)$ induced by the correspondence of these non-singular geodesics is a strictly increasing map with respect to the angular order on these spaces. A strictly increasing map defined almost everywhere is continuous almost everywhere, and by applying the same argument with the roles of $p$ and $p_i'$ reversed we see that the map is in fact extendable to a continuous function from the circle of directions at $p_i'$ to the circle of directions at $p$.

Using this, we see that almost every non-singular $\tilde g_2$-geodesic through $p_1'$ is a bounded distance from a non-singular $\tilde g_2$-geodesic through $p_2'$. These pairs of geodesics must bound flat strips, but by Proposition \ref{prop:flat strip volume} we know this is impossible.

\end{proof}

Let $\tilde U_i = \{p\in (\tilde S\setminus\tilde P,\tilde g_i): \kappa_{\tilde g_i}(p)<0\}$ be the set of non-cone points where the $\tilde g_i$-curvature is strictly negative. This is an open subset of $\tilde S$.

\begin{lem}[\cite{cff}, Lemma 2.4]
$\tilde U_1$ is in the domain $\tilde D_1$ of $F_{neg}$. 
\end{lem}

\begin{proof}
Let $p\in \tilde U_1$. Note that no $\tilde g_1$-geodesic through $p$ bounds a flat strip. Pick any two non-singular geodesics $l_1$ and $l_2$ through $p$. Then $l_1$ and $l_2$ are in $\hat{\mathcal{G}}^\circ_{\tilde g_1}$. Let $l_i'$ be corresponding $\tilde g_2$-geodesics, and let their intersection be $p'$ -- it is unique by Proposition \ref{prop:angles}. We also note that $l_i'$ are in $\hat{\mathcal{G}}^\circ_{\tilde g_2}$ by Corollary \ref{cor:good geodesics}.

Now let $l_3'$ be a geodesic in $\hat{\mathcal{G}}^\circ_{\tilde g_2}$ through $p'$. Note that almost every geodesic through $p'$ is in $\hat{\mathcal{G}}^\circ_{\tilde g_2}$, since, by the argument of Proposition \ref{prop:flat strip volume} only countably many flat strips can pass through any point.

Now consider the geodesic triangle formed by $l_1, l_2,$ and $l_3$. By the angle correspondence of Proposition \ref{prop:angles}, it has angle sum $\pi$, and since $\tilde g_1$ has negative curvature at $p$, the triangle must degenerate to the point $p$ by Gauss-Bonnet with singularities. Thus $l_3$ passes through $p$ and we have proven that $p\sim p'$, as desired.
\end{proof}

\begin{rem}
Note that $F_{neg}$ works in the same way $F_c$ of section \ref{sec:cone} does -- by taking $p$ to the common intersection of $\tilde g_2$-geodesics corresponding to a full measure set of $\tilde g_1$-geodesics passing through $p$.
\end{rem}

\begin{prop}
Let $g_1, g_2\in \mathscr{M}_{npc}(S)$ with the same marked length spectrum. Then the map $F_{neg}$ described above is a $\Gamma$-equivariant bijective isometry between $\tilde U_1$ and $\tilde U_2$.
\end{prop}

\begin{proof}
Let $p_1, p_2 \in \tilde U_1$, and let $k$ be the $\tilde g_1$-geodesic segment between them. Consider the set $\mathcal{G}_{\tilde g_1}(k)$ of all $\tilde g_1$-geodesics intersecting $k$ transversally. $F_{neg}(p_1)$ and $F_{neg}(p_2)$ lie in at most countably many flat strips. So it is not hard to see that all but countably many $\{a,b\} \in \partial_{\tilde g_1}(\mathcal{G}_{\tilde g_1}(k))$ correspond to elements of $\partial_{\tilde g_2}(\mathcal{G}_{\tilde g_2}(k'))$ where $k'$ is the geodesic segment connecting $F_{neg}(p_1)$ and $F_{neg}(p_2)$, and vice versa. Then, just as in the last step of the proof of Theorem \ref{thm:cone geodesics},

\begin{align}
	d_{\tilde g_1}(p_1, p_2) & = L_{g_1}(\partial_{\tilde g_1}(\mathcal{G}_{\tilde g_1}(k))) \nonumber \\
						& = L_{g_2}(\partial_{\tilde g_2}(\mathcal{G}_{\tilde g_2}(k'))) \nonumber \\
						& = d_{\tilde g_2}(F_{neg}(p_1),F_{neg}(p_2)). \nonumber
\end{align}

We now know that $F_{neg}:\tilde U_1 \to \tilde S$ is a metric isometry; it is clearly $\Gamma$-invariant.  By considering this metric isometry applied to short geodesic segments in $\tilde U_1$, it is easy to see that it must in fact be a Riemannian isometry, and a homeomorphism onto its image. Therefore it must take $\tilde U_1$ into $\tilde U_2$. Reversing the roles of the two metrics proves it is bijective, completing the proof.
\end{proof}

%

\section{Building the full isometry}

In sections \ref{sec:cone} and \ref{sec:neg} we built $\Gamma$-equivariant isometries $F_c:cone(\tilde g_1)\to cone(\tilde g_2)$ and $F_{neg}:\tilde U_1 \to \tilde U_2$ between the cone points and points of negative curvature, respectively. We first note that both of these maps are defined in the same way -- by showing that the correspondence between $\tilde g_1$-geodesics and $\tilde g_2$-geodesics takes (a full measure set of) geodesics through a point $p$ to (a full measure set of) geodesics through $F_{-}(p)$. The proof that the maps are isometries is the same in each case. Therefore, these maps can be combined into a single isometry:

\begin{prop}
$F':=F_c \cup F_{neg}: cone(\tilde g_1) \cup \tilde U_1 \to cone(\tilde g_2) \cup \tilde U_2$ is a $\Gamma$-equivariant isometry.
\end{prop}

What remains is to extend this isometry to the set of points of curvature zero. To do so we follow ideas of both Croke-Fathi-Feldman and Bankovic-Leininger.

The following two lemmas will be useful.

\begin{lem}\label{lem:seg intersect}
Let $p_1, p_2, q_1, q_2 \in cone(\tilde g_1) \cup \tilde U_1$. If the $\tilde g_1$-geodesic segments $[p_1, p_2]$ and $[q_1, q_2]$ intersect in their interiors, then the $\tilde g_2$-geodesic segments $[F'(p_1), F'(p_2)]$ and $[F'(q_1), F'(q_2)]$ intersect in their interiors.
\end{lem}

\begin{proof}
Recall that $F'(p)$ is the common intersection point of $\tilde g_2$-geodesics corresponding to almost every $\tilde g_1$-geodesic through $p$. From this it is immediate that any geodesic $l$ through $p_1$ and $p_2$ corresponds to a geodesic through $F'(p_1)$ and $F'(p_2)$. Then the configuration of geodesics in $\hat{\mathcal{G}}^-_{\tilde g_1}$ in the left half of Figure \ref{fig:crossing} must correspond to a configuration as in the right half, where $h'$ and $l'$ intersect between $l'_1$ and $l'_2$. 

\begin{center}
\setlength{\unitlength}{.28mm}
\begin{picture}(440,200)(-220,-90)

\put(-150, 95){$(\tilde S, \tilde g_1)$}
	
\qbezier(-220,0)(-220,80)(-140,80)
\qbezier(-220,0)(-220,-80)(-140,-80)
\qbezier(-60,0)(-60,80)(-140,80)
\qbezier(-60,0)(-60,-80)(-140,-80)

\put(-40,0){\vector(1,0){60}}
\put(-10,10){$F'$}

\put(110,95){$(\tilde S, \tilde g_2)$}

\qbezier(40,0)(40,80)(120,80)
\qbezier(200,0)(200,80)(120,80)
\qbezier(40,0)(40,-80)(120,-80)
\qbezier(200,0)(200,-80)(120,-80)
	
\put(-140,40){\circle*{2}}
\put(-155,45){$q_1$}
\put(-140,-40){\circle*{2}}
\put(-155,-50){$q_2$}
\put(-180,0){\circle*{2}}
\put(-190,-10){$p_1$}
\put(-100,0){\circle*{2}}
\put(-110,-10){$p_2$}

\linethickness{0.35mm}
\put(-140,40){\line(0,-1){80}}
\put(-180,0){\line(1,0){80}}

\linethickness{0.1mm}
\qbezier(-140,40)(-140,60)(-130,79)
\qbezier(-140,-40)(-140,-60)(-150,-79)
\qbezier(-180,0)(-200,0)(-220,10)
\qbezier(-100,0)(-80,0)(-60,10)

\put(-75,8){$l$}
\put(-130,65){$h$}

\qbezier(-211,45)(-140,35)(-69,45)
\qbezier(-211,-45)(-140,-35)(-69,-45)

\put(-190,46){$l_1$}
\put(-190,-53){$l_2$}

\put(150,40){\circle*{2}}
\put(130,-40){\circle*{2}}
\put(60,0){\circle*{2}}
\put(120,10){\circle*{2}}

\put(150,40){\line(-1,-4){20}}
\put(150.5,40){\line(-1,-4){20}}
\put(149.5,40){\line(-1,-4){20}}
\put(60,0){\line(6,1){60}}
\put(62,0){\line(6,1){60}}
\put(58,0){\line(6,1){60}}

\qbezier(150,40)(150,60)(140,78)
\qbezier(130,-40)(130,-60)(140,-78)
\qbezier(60,0)(50,-2)(40,-5)
\qbezier(120,10)(160,15)(200,10)

\qbezier(60,60)(150,37)(195,35)
\qbezier(60,-60)(150,-25)(190,-47)

\put(187,15){$l'$}
\put(150,60){$h'$}
\put(80,59){$l_1'$}
\put(80,-65){$l_2'$}
\put(133,-52){$F'(q_2)$}
\put(152,42){$F'(q_1)$}

\put(45,-15){$F'(p_1)$}
\put(100,15){$F'(p_2)$}

\end{picture}
\begin{figure}[h]
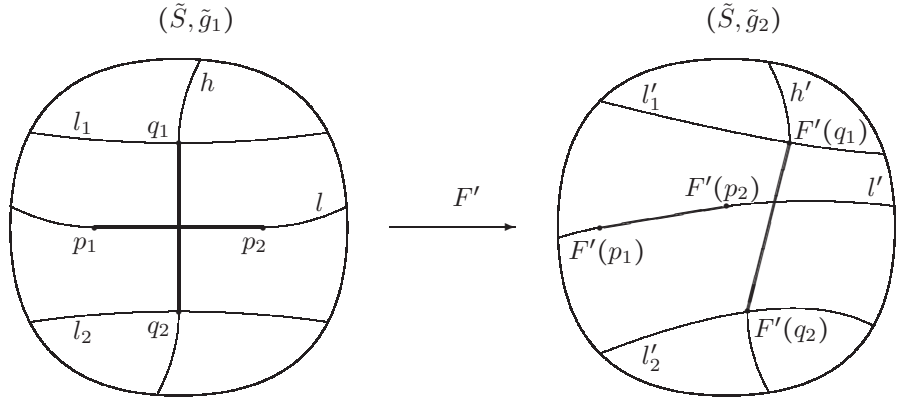

\caption{$F'$ applied to intersecting segments.}\label{fig:crossing}
\end{figure}

\end{center}

\noindent (Note that such a configuration can be drawn, with $l$, $l_1$ and $l_2$ having distinct endpoints, because none of the $p_i, q_i$ can lie in the interior of a flat strip.)

Reversing the roles of $p_i$ and $q_i$, we see that Figure \ref{fig:crossing} is inaccurate as drawn. In fact $h'$ must cross $l'$ between $F'(p_1)$ and $F'(p_2)$, proving the lemma.

\end{proof}

\begin{lem}\label{lem:triangle}
Let $p_1, p_2, p_3$ and $q$ be points in $cone(\tilde g_1)\cup \tilde U_1$ such that $q$ lies in the interior of the $\tilde g_1$-geodesic triangle formed by $p_1, p_2$, and $p_3$. Then $F'(q)$ lies in the interior of the $\tilde g_2$-geodesic triangle formed by $F'(p_1), F'(p_2)$, and $F'(p_3)$.
\end{lem}

\begin{center}
\setlength{\unitlength}{.28mm}
\begin{picture}(440,200)(-220,-100)
	
\qbezier(-220,0)(-220,80)(-140,80)
\qbezier(-220,0)(-220,-80)(-140,-80)
\qbezier(-60,0)(-60,80)(-140,80)
\qbezier(-60,0)(-60,-80)(-140,-80)

\put(-40,0){\vector(1,0){60}}
\put(-10,10){$F'$}

\qbezier(40,0)(40,80)(120,80)
\qbezier(200,0)(200,80)(120,80)
\qbezier(40,0)(40,-80)(120,-80)
\qbezier(200,0)(200,-80)(120,-80)

\put(-140,40){\circle*{4}}
\put(-190,-32){\circle*{4}}
\put(-100,-29){\circle*{4}}
\put(-140,0){\circle*{4}}

\put(120,40){\circle*{4}}
\put(90,-30){\circle*{4}}
\put(150,-30){\circle*{4}}

\qbezier(-213,-40)(-140,-10)(-67,-40)
\qbezier(-210,-48)(-140,0)(-131,80)
\qbezier(-74,-53)(-140,0)(-149,80)

\qbezier(-220,10)(-140,-10)(-60,10)

\put(-140,-40){$l$}
\put(-200,10){$h$}
\put(-140,-10){$q$}

\qbezier(47,-40)(120,-15)(193,-40)
\qbezier(64,-64)(120,0)(126,80)
\qbezier(175,-65)(120,0)(113,79)
\qbezier(40,0)(120,20)(186,52)

\put(60,10){$h'$}
\put(120,-41){$l'$}

\end{picture}

\begin{figure}[hh]
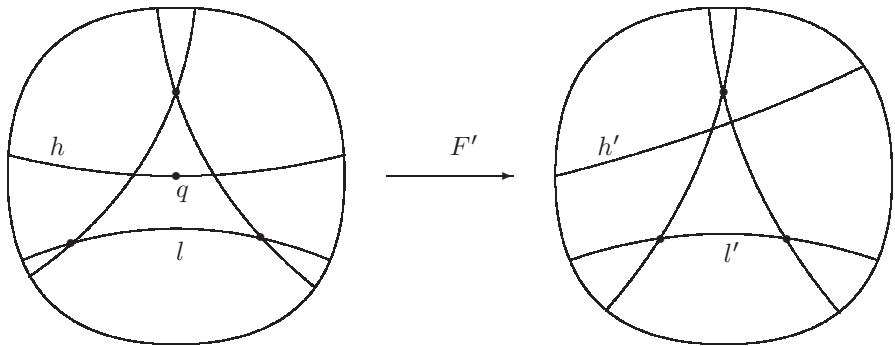

\caption{$F'$ applied to geodesic triangles.}\label{fig:triangle}
\end{figure}

\end{center}

\begin{proof}
Figure \ref{fig:triangle} shows that $F'(q)$ must lie in the open half-space bounded by $l'$ containing $\Delta(F'(p_1),F'(p_2),F'(p_3))$. Running this argument for all three sides of $\Delta(p_1,p_2,p_3)$ proves the lemma. (Again, the configuration in this figure can be drawn since $q$ does not belong to the interior of any flat strip and so $h$ can be taken to have distinct endpoints from $l$.)
\end{proof}

Note that we can assume that $cone(g_1)$ (and hence $cone(g_2)$ by Theorem \ref{thm:cone geodesics}) are non-empty -- otherwise Croke-Fathi-Feldman applies. Then we may take a $\Gamma$-invariant geodesic triangulation $\tilde \tau$ of $(\tilde S,\tilde g_1)$ such that the vertices are precisely the cone points. (This is an easy construction). The lemmas above show that $F'$ sends this triangulation to a triangulation of $(\tilde S, \tilde g_2)$:

\begin{lem}
Let $F'(\tilde \tau)$ be the collection of $\tilde g_2$-geodesic segments $\{[F'(p_1),F'(p_2)]: [p_1,p_2]\in\tau\}.$ Then $F'(\tilde \tau)$ is a $\Gamma$-invariant triangulation of $(\tilde S,\tilde g_2)$ with vertex set precisely $cone(\tilde g_2)$. Further, if $q\in cone(\tilde g_1)\cup \tilde U_1$ belongs to a triangle $T$ of $\tilde \tau$, then $F'(q)\in F'(T)$, where $F'(T)$ is the triangle formed by the $F'$ images of the vertices of $T$.
\end{lem}

\begin{proof}
That $F'(\tilde\tau)$ is $\Gamma$-invariant with vertex set $cone(\tilde g_2)$ is clear. That it preserves containment in triangles is an immediate application of Lemma \ref{lem:triangle}.

Every segment in $\{[F'(p_1),F'(p_2)]: [p_1,p_2]\in\tilde \tau\}$ belongs to two 3-cycles in $F'(\tilde \tau)$, thought of as an abstract graph, on opposite sides of that edge, since this is true for $\tau$. By Lemma \ref{lem:seg intersect} these segments do not intersect away from the vertices. Then the geometric realization of this graph in the simply connected $\tilde S$ is an infinite, locally finite, planar graph, all of whose edges belong to two triangles on opposite sides of the edge. Therefore it is a triangulation. 
\end{proof}

To complete the proof of Theorem \ref{main thm} we need to extend $F'$ to a $\Gamma$-equivariant isometry on all of $(\tilde S,\tilde g_1)$.

\begin{proof}[Proof of Theorem \ref{main thm}]

Fix a $\Gamma$-invariant geodesic triangulation $\tilde \tau$ of $(\tilde S, \tilde g_1)$ as above.

Let $\mathcal{T}_0(\tilde \tau)$ be the set of all triangles of $\tilde \tau$ which have $\tilde g_1$-flat interior. Since $F'$ respects containment of points with negative curvature in triangles, the induced map from triangles of $\tilde \tau$ to triangles of $F'(\tilde \tau)$ takes $\mathcal{T}_0(\tilde \tau)$ bijectively to $\mathcal{T}_0(F'(\tilde \tau))$. Let $V_1 = \{ p\in \bar T: T\in \mathcal{T}_0(\tilde \tau)\}$ and $V_2 = \{ p\in \bar T: T\in \mathcal{T}_0(F'(\tilde \tau))\}$. Since $F'$ is an isometry on the vertices of these Euclidean triangles, it can be extended to $F'_0:cone(\tilde g_1)\cup U_1 \cup V_1 \to cone(\tilde g_2)\cup U_2 \cup V_2$ as a Riemannian isometry away from the cone points. It is clear that $F'_0$ is well-defined on any points at the boundary of two such triangles, and that $F'_0$ is $\Gamma$-equivariant.

If $\tilde g_1$ is Euclidean away from cone points, this completes the proof. This is similar to, but not exactly, the argument of Bankovic-Leininger.

Now let $\mathcal{T}_<(\tilde \tau)$ be the set of all triangles of $\tilde \tau$ for which there is some point in $Int(T)$ at which the $\tilde g_1$-curvature is strictly negative. To extend $F'_0$ to these triangles, we use the approach of Croke-Fathi-Feldman.  $\mathcal{T}_<(\tilde \tau)$ is certainly $\Gamma$-invariant and $F'$ takes this collection of triangles to $\mathcal{T}_<(F'(\tilde \tau))$. Pick, for each triangle $T$ in $\mathcal{T}_<(\tilde \tau)$ a point $p^*_T \in Int(T)$ at which $\kappa(p^*_T)<0$ in a $\Gamma$-invariant way. Let $V_T \subset T_{p^*_T}\tilde S$ be such that $\exp_{p^*_T}^{\tilde g_1}(V_T) =  \bar T$. This set exists since there are no cone points in the interior of $T$ and the exponential map is injective as we are in nonpositive curvature. Then, following Croke-Fathi-Feldman, we define

\begin{equation*}
	\begin{array}{rrcl}
	 F_T: & (T,\tilde g_1) & \to & (\tilde S,\tilde g_2)  \\
	 & p & \mapsto & (\exp_{F'_0(p_T^*)}^{\tilde g_2})\circ (DF'_0)_{p^*_T}\circ (\exp_{p^*_T}^{\tilde g_1})^{-1}(p). 
	 \end{array}
\end{equation*}

First, we claim that $F_T$ extends $F'_0$, that is, if $q\in \bar T \cap(cone(\tilde g_1) \cup \tilde U_1)$, then $F_T(q) = F'_0(q)$. Let $c(t)$ be the unit-speed $\tilde g_1$-geodesic connecting $p^*_T$ to $q$ with $c(0)=p^*_T$. Then $d_{\tilde g_1}(c(t),q) = d(p^*_T,q)-t$. For sufficiently small $|t|$, $c(t)$ lies in $\tilde U_1$ since $c(0)$ does. Thus, for sufficiently small $|t|$, we can use the fact that $F'_0$ is an isometry and we have

\[ d_{\tilde g_2}(F'_0(c(t)), F'_0(q)) =  d_{\tilde g_2}(F'(p_T^*),F'(q))-t.\]

\noindent This implies that $F'_0(c(t))$ lies along the geodesic segment connecting $F'_0(p^*_T)$ and $F'_0(q)$ for at least small values of $|t|$. In particular, $(DF'_0)_{p^*_T}(\dot c(0))$ is tangent to the geodesic from $F'_0(p^*_T)$ to $F'_0(q)$. Then, using the definition of $F_T$ and the fact that $F'_0$ is an isometry, it is easy to see that $F_T(q) = F'_0(q)$.

Second, $F_T$ must preserve curvature, since we now know it agrees with the isometry $F'_0$ on points of negative curvature and cone points, and all other points have zero curvature.

Third, we claim that $F_T|_{Int(T)}$ is a Riemannian isometry onto its image. At points in $\tilde U_1$ this follows from the fact that it is a (metric) isometry. Let $q$ be a point in $\tilde U_1^c$ and $v\in T_q\tilde S$. If $v$ is tangent to the geodesic connecting $p^*_T$ and $q$, then it is easy to see from the definition of $F_T$ that $|(DF_T)_q(v)|_{\tilde g_2} = |v|_{\tilde g_1}$. Now assume that $v$ is normal to the geodesic from $p^*_T$ to $q$. There is a unique Jacobi field along this geodesic with value 0 at $p^*_T$ and value $v$ at $q$. Since any Jacobi field arises from a variation of geodesics, $DF_T$ takes this Jacobi field to a Jacobi field along the geodesic from $F_T(p^*_T)$ to $F_T(q)$. By the second point above, the curvatures along the geodesic segments $[p^*_T,q]$ and $[F'(p^*_T),F'(q)]$ agree, so by the Jacobi equation, $|(DF_T)_q(v)|_{\tilde g_2} = |v|_{\tilde g_1}$. This proves the claim.

Fourth, we claim that $F_T$ has its image in the geodesic triangle $F'(T)$ from the triangulation $F'(\tilde \tau)$. Since $F_T$ is a Riemannian isometry on its interior, and since the geodesics bounding $T$ can be approached by sequences of geodesics in $Int(T)$, we see that $F_T$ takes $T$ to a geodesic triangle. From the first claim, we know $F_T$ takes the vertices of $T$ to the vertices of $F'(T)$. But there is only one geodesic triangle with these vertices -- namely $F'(T)$. Reversing the roles of $\tilde g_1$ and $\tilde g_2$ shows that $F_T$ is bijective onto its image.

We now let $F = \bigcup_{T\in \tilde \tau} F_T$. It is clear from its construction that $F$ is $\Gamma$-equivariant and extends $F'_0$. The fact that each $F_T$ maps the geodesic edges to the geodesic edges of $F'(T)$ preserving arc-length implies that $F$ is well-defined along the edges of $\tilde \tau$. The fact that $F_T$ is a Riemannian isometry on the open dense subset consisting of the interiors of all the triangles in $\tilde \tau$ and an isometry on the cone point sets, together with the fact that distance for $\tilde g_1$ or $\tilde g_2$ is realized by lengths of shortest paths then implies that $F$ is an isometry easily.

To prove that the map $F$ induces on $(S, g_1) \to (S,g_2)$ is isotopic to the identity, we note that by $\Gamma$-equivariance, the extension of $F$ to $\partial^\infty(\tilde S)$ is the identity, so the map $F$ induces on $\Gamma$ is the identity, proving the result.

This completes the proof of Theorem \ref{main thm}.

\end{proof}

\bibliographystyle{alpha}
\bibliography{biblio}

\end{document}